\documentclass{amsart}

\usepackage{amsthm,amsmath,amssymb}

\newcommand{\R}{\mathbb{R}}

\newcommand{\C}{\mathbb{C}}

\newcommand{\kahler}{K\"ahler }

\newtheorem{thm}{Theorem}[section]
\newtheorem{cor}[thm]{Corollary}
\newtheorem{lemma}[thm]{Lemma}

\newtheorem{prop}[thm]{Proposition}

\newtheorem{rmk}{Remark}

\newtheorem{Def}[thm]{Definition}

\title{THE GLOBAL EXISTENCE AND CONVERGENCE OF THE CALABI FLOW ON $\mathbb{C}^n/\mathbb{Z}^n+i \mathbb{Z}^n$}
\date{December 10, 2011}
\author{Renjie Feng \quad Hongnian Huang}

\thanks{The research of the second named author is financially supported by CIRGET (Centre de
recherche en g\'eom\'etrie et topologie), Montreal, Canada and FMJH
(Fondation math\'ematique Jacques Hadamard), Paris, France.} 

\begin{document}

\maketitle

\begin{abstract}
In this note, we study the long time existence of the Calabi flow on $X = \mathbb{C}^n/\mathbb{Z}^n+i \mathbb{Z}^n$. Assuming the uniform bound of the total energy, we establish the non-collapsing property of the Calabi flow by using Donaldson's estimates and Streets' regularity theorem. Next we show that the curvature is uniformly bounded along the Calabi flow on $X$ when the dimension is 2,  partially confirming Chen's conjecture.  Moreover, we show that the Calabi flow exponentially converges to the flat K\"ahler metric for arbitrary dimension if the curvature is uniformly bounded, partially confirming Donaldson's conjecture.
\end{abstract}

\section{Introduction}

The Calabi flow was invented by Calabi \cite{Ca1} to search for the canonical metrics in a given K\"ahler class. Let $\varphi$ be a K\"ahler potential and $S$ be the scalar curvature of $\varphi$, its equation is
\begin{equation}\label{calabi}
\frac{\partial \varphi}{\partial t} =S - \underline{S},
\end{equation}
where $\underline{S}$ is the average of the scalar curvature. Since it is a 4th order parabolic equation, its long time existence and convergence are hard to study. The Riemann surface case is settled down by Chrusci\'el \cite{Chr} and reproved by Chen \cite{Chen1}. The study of the Calabi flow on the ruled manifolds is elaborated by Guan \cite{Gu} and Sz\'ekelyhidi \cite{Sz}. Later, Chen and He prove the short time existence of the Calabi flow in \cite{ChenHe}. In the same paper, they prove that the obstruction of the long time existence of the Calabi flow is the Ricci curvature. Furthermore, they establish the stability property of the Calabi flow near a cscK metric. The stability property is generalized by Zheng and the second named author in \cite{HZ} for the case of extremal metrics.  Tosatti and Weinkove \cite{TW} also prove the stability the Calabi flow when the first Chern class $c_1 = 0$ or $c_1 < 0$. The stability problem is further studied in Chen and Sun's work \cite{CS}, they prove that constant scalar curvature K\"ahler metric ``adjacent" to a fixed K\"ahler class is unique up to isomorphism.

The long time existence problem of the Calabi flow largely remains open. Assuming the long time existence, Donaldson describes the limiting behavior of the Calabi flow in \cite{D5}. Sz\'ekelyhidi \cite{Sz2} shows that if the Calabi flow exists for all time in toric varieties, then the infimum of the Calabi energy is equal to the supremum of the normalized Futaki invariant over all destabilizing test-configurations, partially confirming Donaldson's conjecture in \cite{D6}. 

The global convergence problem of the Calabi flow also largely remains open. An application of the global convergence of the Calabi flow is to solve a conjecture proposed by Apostolov, Calderbank, Gauduchon and T{\o}nnesen-Friedman \cite{ACGT} : A projective bundle $(M,J) = P(E)$ over a compact curve of genus $\geq 2$ admits an extremal K\"ahler metric in some K\"ahler class if and only if $E$ decomposes as a direct sum of stable sub-bundles.

One of the methods in studying the long time existence problem is the blow-up analysis. It is firstly adopted in Chen and He's work \cite{ChenHe2}. They establish the following {\bf weak regularity theorem}: Suppose the $L^\infty$ norm of Riemann curvature tensor of the Calabi flow is bounded by 1 in the time interval $[-1,0]$, then
\begin{eqnarray}
\label{weak}
\int_X |\nabla^k Rm(0,x)|^2 d \omega < C \left(n, k, \int_X |Rm(-1,x)|^2 d \omega \right).
\end{eqnarray}

\begin{rmk}
Streets also obtains a similar result in \cite{St2}.
\end{rmk}

When the Calabi energy is small in certain toric Fano surfaces, Chen and He are able to obtain a uniform Sobolev constant along the Calabi flow \cite{ChenHe3}. Hence they derive the {\bf regularity theorem} of the Calabi flow: Suppose the $L^\infty$ norm of Riemann curvature tensor of the Calabi flow is bounded by 1 in the time interval $[-1,0]$, then
\begin{eqnarray}
\label{reg}
\max_{x \in X} |\nabla^k Rm(0,x)| < C(n, k).
\end{eqnarray}

\begin{rmk}
The regularity theorem is called Shi's estimate in the Ricci flow \cite{Shi}.
\end{rmk}

After obtaining the uniform bound of Sobolev constant and the regularity theorem, Chen and He rule out the singularities along the Calabi flow and show that the Calabi flow converges to an extremal metric in the Cheeger-Gromov sense. This result gives us a better understanding of Chen's conjecture (see e.g. \cite{ChenHe3}) :
$$\bold{Conjecture}: \,\, \text{The Calabi flow exists for all time}.$$

Motivated by Donaldson's work in \cite{D1} \cite{D2} \cite{D3} \cite{D4}, the second named author studies the classification of singularities of the Calabi flow on toric varieties by assuming the total energy bound, the regularity theorem and the non-collapsing property \cite{H1}. Later, Streets proves the regularity theorem for the Calabi flow \cite{St}. Then the remaining obstacle for the long time existence of the Calabi flow on a toric variety is the non-collapsing property of the Calabi flow.

The Calabi flow on toric varieties is a parabolic version of the linearized Monge-Amp\`ere equation. The linearized Monge-Amp\`ere equation is studied in the work of Caffarelli and Guti\'errez \cite{CG}. Their work has been used in Trudinger-Wang's solution of the Bernstein problem \cite{TrWa} and the affine Plateau problem \cite{TrWa1}. In Donaldson's work on the existence of cscK metrics on toric surfaces, he also uses Caffarelli and Guti\'errez's work to obtain the interior regularity of his continuous method \cite{D2} and the $M$-condition near the boundary in order to resolve the non-collapsing issue \cite{D4}.  Caffarelli and Guti\'errez's work also finds applications in Chen Li and Sheng's work on the existence of extremal metrics on toric surfaces \cite{CLS}. Sz\'ekelyhidi and the first named author apply the ideas of Trudinger-Wang and Donaldson to solve the Abreu's equation on Abelian varieties \cite{FS}.

The difficulty of the long time existence of the Calabi flow is to show the non-collapsing property, i.e., the injectivity radius has a uniform lower bound in the blow-up analysis. More details can be found in \cite{BB}, \cite{CZ}, \cite{KJ} and \cite{MT} where authors explain how to use the non-collapsing property of Ricci flow to classify the singularities of the Ricci flow in a 3-manifold.

\subsection{Main results}
For simplicity, we only consider the long time existence and the global convergence of the Calabi flow on $X = \mathbb{C}^n / \Lambda $, where $\Lambda = \mathbb{Z}^n + i \mathbb{Z}^n$. There is a natural $T^n$ action on $X$ via the translation in the Lagrangian subspace $i \R^n\subset\C^n$. Let $\omega_0$  be a flat metric. We consider the space of $T^n$-invariant \kahler metric (Section 2): $$\mathcal H_{T^n}=\{\phi\in C_{T^n}(X): \omega_\phi=\omega_0+\partial\bar\partial\phi>0 \}.$$   Then we can prove the following non-collapsing theorem along the Calabi flow in the space $\mathcal H_{T^n}$.

\begin{thm}
\label{non-collapsing}
Let $\omega_{\phi_{(-1)}}\in \lambda \mathcal H_{T^n} $ be an initial metric, where $\lambda > 1$ is an arbitrary rescaling factor. Suppose that:
\begin{itemize}

\item The Calabi flow exists for $t \in [-1,0]$ in $ \lambda \mathcal H_{T^n} $ and the $L^\infty$ norm of Riemann curvature tensor of the Calabi flow is uniformly bounded by 1 on $X\times [-1,0]$.

\item  The total energy is bounded at the end point $t=0$, i.e.,
$$
\int_X |Rm(0, x)|^n \  \omega_{\phi_{0}}^n < C,
$$
where $C$ is a positive constant.

\item There is a constant $M$ such that the Legendre transform of the K\"ahler potential of $\omega_{\phi_{0}}$ satisfies the $M$-condition.

\item $|Rm(0,x)| = 1$  for some $x \in X$.
\end{itemize}

Then the injectivity radius of $x$ at time $t=0$ is bounded from below by $C_1$ depending only on $n, M$ and $C$.
\end{thm}

Next we obtain the following long time existence result of the Calabi flow which partially confirms Chen's conjecture.

\begin{thm}
\label{long}
In  dimension 2, given any initial data in $\mathcal H_{T^n}$, the Calabi flow exists for all time in $\mathcal H_{T^n}$ and the curvature is uniformly bounded along the flow.
\end{thm}

For the global convergence, Donaldson has the conjecture that: If the Calabi flow exists for all time and there exists a cscK metric in the K\"ahler class, then the Calabi flow converges to a cscK metric \cite{D5}. Berman \cite{B1} proves that the Calabi flow converges to a K\"ahler Einstein metric in the weak topology of currents if the Calabi flow exists for all time. Our following result confirms Donaldson's conjecture in $X$ for arbitrary dimension.

\begin{thm}
\label{global}
If the Calabi flow exists for all time in $\mathcal H_{T^n}$ and the curvature is uniformly bounded, then it converges to $\omega_0$ which is a flat metric.
\end{thm}

{\bf Acknowledgment: } The authors would like to express their gratitude to the anonymous referee for his numerous suggestions. The second named author is grateful for the consistent support of Professor Xiuxiong Chen, Pengfei Guan, Vestislav Apostolov and Paul Gauduchon. He also benefited from the discussion with Professor Shing-Tung Yau during his visit at Harvard University. He wants to thank Joel Fine, Si Li and Jeffrey Streets for useful conversations. Part of this work was done while the second named author was visiting the Northwestern University, he would like to thank Professor Steve Zelditch for his warm hospitality.

\section{Abelian varieties}\label{tori}

Let $X = \mathbb C^n/\Lambda$ where $\Lambda=\mathbb Z^n+i \mathbb Z^n$. We write each point as $z=\xi+i\eta$, where $\xi$ and $\eta \in \R^n$ and can be viewed as the periodic coordinates of $X$. Let
$$\omega_0=\frac{\sqrt{-1}}{2}\sum_{\alpha=1}^n d z_\alpha\wedge d\bar z_\alpha=\sum_{\alpha=1}^n d\xi_\alpha\wedge d\eta_\alpha$$
be the standard flat metric with associated local K\"ahler potential $\frac12|z|^2$. The group $T^n$ acts on $X$ via translation in $\eta$ variable in the Lagrangian subspace $i\R^n\subset \C^n$, thus we can consider the space of torus invariant  K\"ahler metrics in the fixed class $[\omega_0]$: $$\mathcal H_{T^n}=\{\phi\in C_{T^n}^\infty(M): \omega_\phi=\omega_0+\frac{\sqrt{-1}}{2}\partial\bar\partial\phi>0\}$$
Functions invariant under the translation of $ T^n$ are independent of $\eta$, so they are smooth functions on $X/T^n\cong T^n$, i.e., $\phi(\xi)$ is a periodic and smooth function on $\mathbb{R}^n$.  Without loss of generality, we can assume that the fundamental domain for the periodicity of $\phi$ is $[-\frac12,\frac 12]^n$. We can write the local K\"ahler potential in $\mathcal H_{T^n}$ in complex coordinates as
 $$v(\xi)=\frac12|\xi|^2+\phi(\xi)$$
 and the scalar curvature is
 $$S=-\sum_{i,j}v^{i\bar j}\log[\det(v_{a\bar b})]_{i\bar j}$$
where $v$ is a convex function on $\R^n$ since it is a local K\"ahler potential. Then we can take the Legendre transform of $v$, with dual coordinate $x=\nabla v(\xi)$. In fact, $x$ induces a Lie group moment map: $X\to T^n$. The transformed function $u(x)$ is defined by
  $$u(x)+v(\xi)=x \cdot \xi.$$

The image of the Lie group moment is isomorphic to $X/T^n\cong T^n$. We denote $P=[-\frac12,\frac 12]^n$ as the fundamental domain of $T^n$. Let $\underline{u}=\frac 12|x|^2.$ One can check that $u-\underline{u}$ is a periodic function in $\mathbb{R}^n$ with fundamental domain $P$.

  A calculation in \cite{A1} gives
 $$ S=-\sum_{i,j}\frac{\partial^2 u^{ij}(x)}{\partial x_i\partial x_j}$$
 which is called Abreu's equation, i.e., the expression of scalar curvature under symplectic coordinates. Notice that in our case, the average of $S$ is $0$, thus we can rewrite
  the Calabi flow in terms of Abreu's equation as
  \begin{equation}
\frac{\partial u}{\partial t} = \sum_{i,j}\frac{\partial^2 u^{ij}(x)}{\partial x_i\partial x_j}
\end{equation}
where we use the fact that
$\frac {\partial v(t,\xi)}{\partial t}=-\frac{ \partial u(t,x)}{\partial t}$ \cite{Gu1}. In fact, by the proof of the short time existence of the Calabi flow in \cite{ChenHe}, if the initial metric is in $\mathcal H_{T^n}$, then the Calabi flow will stay in $\mathcal H_{T^n}$ for a short time.

\section{Calabi flow and $M$-condition}
First, we want to introduce Donaldson's $M$-condition which is crucial in controlling the injectivity radius.

For any line segment $\overline{p_0  p_3} \subset P$, let $p_1, p_2 \in P$ be two points in $P$ such that the lengths of $\overline{p_0  p_1}, \overline{p_1  p_2}, \overline{p_2  p_3}$ are the same. Let $\nu$ be the unit vector parallel to the vector $p_3 - p_0$.  We say that $u$ satisfies the $M$-condition on  $\overline{p_0 p_3}$ if
$$
|\nabla_\nu u(p_1) - \nabla_\nu u(p_2)| < M.
$$

\begin{Def}
If for any line segment $l \subset P$, $u$ satisfies the $M$-condition on $l$, then we say that $u$ satisfies the $M$-condition on $P$.
\end{Def}

The goal of this section is to prove the following proposition:

\begin{prop}
\label{Calabi-M}
The $M$-condition is preserved under the Calabi flow.
\end{prop}
To achieve this goal, we prove $C^0$ and $C^1$ bounds of the solution of the Calabi flow. Thus, there is a uniform constant $M$ such that the $M$-condition holds for all time.

In Calabi and Chen's work \cite{CC}, they show that the Calabi flow decreases the distance. We will reproduce this result in our settings. 

\begin{prop}
The $L^2$ norm of $u(t, x)-\underline{u}$ is decreasing along the Calabi flow.
\end{prop}

\begin{proof}
Let $\nu$ be the outward normal vector along the boundary of $P$ and $ds$ be the boundary measure. We have
\begin{eqnarray*}
& &\frac{\partial}{\partial t} \int_P (u(t,x)-\underline{u})^2 \ dx \\
&=& 2 \int_P (\underline{S}-S(t)) (u(t,x)-\underline{u}) dx \\
&=& 2 \int_P \left(u^{ij}_{\ ij}(t,x) - \underline{u}^{ij}_{\ ij} \right) (u(t,x)-\underline{u}) \ dx \\
&=& 2 \int_{\partial P} \left(u^{ij}_{\ i}(t,x) - \underline{u}^{ij}_{\ i} \right)\nu_j (u(t,x)-\underline{u}) \ ds - 2 \int_P \left(u^{ij}_{\ i}(t,x) - \underline{u}^{ij}_{\ i} \right) (u_j(t,x)-\underline{u}_j) \ dx \\
&=& - 2 \int_P \left(u^{ij}_{\ i}(t,x) - \underline{u}^{ij}_{\ i} \right) (u_j(t,x)-\underline{u}_j) \ dx \\
&=& -2 \int_{\partial P} \left(u^{ij}(t,x) - \underline{u}^{ij} \right) \nu_i (u_j(t,x)-\underline{u}_j) \ ds + 2\int_P \left(u^{ij}(t,x) - \underline{u}^{ij} \right) (u_{ij}(t,x) - \underline{u}_{ij}) \ dx \\
&=& 2\int_P \left(u^{ij}(t,x) - \underline{u}^{ij} \right) (u_{ij}(t,x) - \underline{u}_{ij}) \ dx \\ &\leq & 0.
\end{eqnarray*}

The boundary integrals vanish due to the fact that $u-\underline{u}$ is a periodic function on $\mathbb{R}^n$ with fundamental domain $P$. For the last step, notice that
$$
\sum_{ij} (u^{ij}(t,x) - \underline{u}^{ij}) (u_{ij}(t,x) - \underline{u}_{ij}) = Trace ( (u^{ij}(t,x) - I) (u_{ij}(t,x) - I)).
$$
Thus at each point we can choose an orthonormal basis such that $(u_{ij}(t,x)) = diag (\lambda_1, \ldots, \lambda_n)>0$ since $u$ is convex, then
$$
\sum_{ij} (u^{ij}(t,x) - \underline{u}^{ij}) (u_{ij}(t,x) - \underline{u}_{ij}) = \sum_{i} (\lambda_i -1 ) (\lambda_i^{-1} -1) \leq 0.
$$
\end{proof}

As an immediate corollary, we have
\begin{cor}
The $L^2$ norm of $u(t,x)$ depends only on the initial metric and is bounded independent of $t$.
\end{cor}

\begin{proof}
\begin{eqnarray*}
& & \int_P u^2(t,x) \ dx
\leq 2 \int_P \underline{u}^2 + (u(t,x)-\underline{u})^2 \ dx \leq 2 \int_P \underline{u}^2 + (u(0)-\underline{u})^2 \ dx \leq C_0.
\end{eqnarray*}
\end{proof}

Then we have the following observation:
\begin{prop}
The $L^\infty$ norm of $u(t,x)$ in $P$ depends only on the initial metric and is bounded independent of $t$.
\end{prop}

\begin{proof}
By the periodicity of $u(t,x) - \underline{u}(x)$, we can do the estimates in larger domains than $P$: we control the upper bound of $u(t,x)$ in $[-2,2]^n$ and the lower bound of $u(t,x)$ in $[-1,1]^n$. Thus we are able to control the gradient of $u(t,x)$ in $P$.

 Notice that the $L^2$ norm of $u(t,x)$ in domain $[-3,3]^n$ is bounded $C_1$ which depends only on the initial metric and is independent of $t$. Let us temporarily suppress the $t$ variable, we will write $u(t,x)$ as $u(x)$ in the proof. First, we prove that $u(x)$ has an upper bound. Suppose not, then the maximum of $u(x)$ in $[-2,2]^n$ reaches at one of its vertices. Without loss of generality, we can assume that the maximum of $u(x)$ is reached at $(2, \ldots, 2)$. We consider the supporting plane of $u(x)$, i.e. $ l(x)$, at $(2, \ldots, 2)$. Then it is easy to see that there is an area larger than 1 such that the value of $l(x)$ in this area is greater or equal to its value at $(2, \ldots, 2)$. Since the value of $u(x)$ is always greater than $l(x)$, we conclude that the $L^2$ norm of $u(x)$ is greater than $C_1$. A contradiction.

Next we show that $u(x)$ is bounded from below in $[-1,1]^n$. Let $x$ be the point in $[-1,1]^n$ reaching the minimum of $u(x)$ in $[-1,1]^n$.  Notice that $u(x)$ is bounded from above along the boundary of $[-2,2]^n$. If $u(x)$ is very negative, then there is an area larger than $2^n$ such that the value of $u(x)$ in this area is less than $u(x)/3$. It also contradicts to the fact that the $L^2$ norm of $u(x)$ is bounded by $C_1$.
\end{proof}

Since $u$ is convex, we have

\begin{cor}
The derivative of $u(t,x)$ with respect to $x$ in $P$ is  bounded by $C_2$ which depends only on the initial metric and is  independent of $t$.
\end{cor}

Now we give a proof of Proposition (\ref{Calabi-M}).
\begin{proof}
It is easy to see that if the derivative of $u$ is bounded, then the $M$-condition holds automatically by its definition.
\end{proof}

\section{Non-Collapsing}
In this section, we apply Donaldson's estimates and the regularity theorem to obtain the lower bound of the injectivity radius in Theorem (\ref{non-collapsing}). Since the estimates are done in $t=0$, we will write $u(x)$ instead of $u(0,x)$ for convenience. Also we write $P$ instead of $\lambda P$ in this section.

Notice that although the evolution equation of the Riemann curvature operator of the Calabi flow is expressed in terms of bisectional curvatures, Streets' estimates which are done in terms of sectional curvatures can still go through:

\begin{enumerate}
\item The integral estimates in \cite{St2} can go through because
$$
\frac{1}{2} \triangle_d = \triangle_{\partial} = \triangle_{\bar{\partial}}
$$
in K\"ahler manifolds.

\item To do the analysis in \cite{St}, we not only need to lift up the metric $g$ to the tangent bundle at $x$, i.e. $T_x X$, but also we need to lift up the holomorphic structure $J$ and the K\"ahler form $\omega$. This can also be done.
\end{enumerate}

Donaldson proves the following lemma in dimension 2. With some modifications, the second named author generalizes Donaldson's result to any dimension.

\begin{lemma}
If the square of the Riemannian curvature norm $$|Rm|^2(x) = \sum_{i,j,k,l} u^{ij}_{\ kl} u^{kl}_{\ ij}(x) \leq 1$$ pointwisely and there is a constant $M$ such that $u$ satisfies the $M$-condition, then for any point $x \in P$, we have
$$
\left(u_{ij}(x) \right) < C I_n,
$$
where $C$ is a constant depending only on $M$.
\end{lemma}

\begin{proof}
See Lemma 4 in \cite{D3} and Lemma 4.4 in \cite{H1}.
\end{proof}

\begin{rmk}
Once we have the upper bound of $(D^2 u)$, we can give a proof of the regularity theorem from the weak regularity theorem, as shown in the Appendix A. 
\end{rmk}

The following lemma which obtains the lower bound of $(u_{ij})$ at one point is established by Donaldson in dimension 2. The second named author generalizes it to higher dimensions.
\begin{lemma}
\label{lower}
Suppose in $P$, the $L^\infty$ norm of Riemannian curvature $$|Rm|_{L^\infty} \leq 1 $$
and
$$
\int_P |Rm(x)|^n \ dx < C_1.
$$
If there is a point $x \in P$ such that $|Rm|(x) = 1$. Then the regularity theorem tells us that
$$
\left(u_{ij}(x) \right) > C_2 I_n,
$$
where $C_2$ depends only on $M, C_1$ and $n$.
\end{lemma}

\begin{proof}
See Proposition 11 in \cite{D3} and Section 5 in \cite{H1}.
\end{proof}

By applying Donaldson's estimates, we can control the lower bound of $(u_{ij}(y))$ for any other point $y$.
\begin{lemma}
Suppose in the polytope $P$, $|Rm|_{L^\infty} \leq 1$. For any point $y \in P$, let the Riemannian distance between $x$ and $y$ be $d$.  We have
$$
(u_{ij}(y)) \geq \frac{1}{e^{2d}} (u_{ij}(x)).
$$
\end{lemma}
\begin{proof}
We apply Lemma 7 in \cite{D3}. Since we do not have boundaries in our case, we can let the boundary distance $\alpha$ go to $\infty$. Hence we obtain the result.
\end{proof}

Notice that if $(u_{ij})$ is bounded from above, then the geodesic distance is bounded above by the Euclidean distance multiplying a constant. Thus we obtain the lower bound of $(u_{ij})(y)$ depending on the Euclidean distance between $y$ and $x$.

Once we obtain the upper bound and lower bound of $(u_{ij})$ of points around $x$. Applying Donaldson's arguments in Lemma 11 of \cite{D3}, we see that the injectivity radius at $x$ is bounded from below. For reader's convenience, we repeat his arguments here.

\begin{proof}
{\em (Theorem (\ref{non-collapsing}))}.
Since we already control the upper bound and lower bound of $(D^2 u)$, we only need to control the injectivity radius of $P \times \mathbb{R}^n$. Moreover, we only need to consider cut points because the curvature is bounded. Applying Lemma 8 of \cite{D3} ( Lemma 4.9 of \cite{H1}), we conclude that the Euclidean metric can be comparable with the Riemannian metric. Then Lemma 10 of \cite{D3} shows that the injectivity radius at $x$ has a lower bound.
\end{proof}

\section{Singularity analysis}
Since we have controlled the $M$-condition along the Calabi flow, thus we are ready to prove Theorem (\ref{long}) by the blow-up analysis. First, we study the formation of singularities of the Calabi flow under the assumption that the total energy, i.e.,
$$
\int_P |Rm(t, x)|^n dx
$$
is uniformly bounded along the flow.

By Chen and He's result \cite{ChenHe}, we know that if the Calabi flow cannot extend over time $T$, then there is a sequence of times $t_i \rightarrow T$ and a sequence of  points $p_i \in P$ such that $|Ric|(t_i,p_i) \rightarrow \infty$. Since we are dealing with the global convergence, we also need to rule out the case that there is a sequence of points $(t_i, p_i)$ where $t_i$ may approach to $\infty$ such that $|Rm(t_i,p_i)|\rightarrow \infty$. We will prove this by the contradiction arguments.

Suppose not, then without loss of generality, we can assume that
$$
|Rm(t_i,p_i)| = \max_{t\leq t_i, p \in P} |Rm(t,p)|.
$$
Denote $\lambda_i = |Rm(t_i,p_i)|$, rescaling the original Calabi flow $u(t)$ by $\lambda_i$, i.e.,
\begin{eqnarray*}
P_i &=& \lambda_i P,\\
u_i(t,x) &=& \lambda_i u(\frac{t-T_i}{\lambda_i^2}, \frac{x-p_i}{\lambda_i}).
\end{eqnarray*}
Then we get a sequence of the Calabi flows $u_i(t)$ such that $|Rm_i(0,0)|=1$ and
$$
\max_{t \leq 0,\ p \in P_i} |Rm|(t,p) = 1.
$$
We try to show that a subsequence of the Calabi flow converges to a limiting Calabi flow and the limiting Calabi flow cannot exist. Then we can conclude that the Riemannian curvature is uniformly bounded along the Calabi flow.

We apply the regularity theorem now. Notice that the first Euclidean derivative of $u$ is unchanged by rescaling, thus the $M$-condition is preserved under the rescaling. Also the total energy is preserved under the rescaling because the contribution from the $|Rm|^n$ cancels with the contribution from the volume. Thus we obtain the non-collapsing property of $(P_i, u_i(0))$. By \cite{D2}, Abreu's equation can be rewritten as
$$
U^{ij} \left( \frac{1}{\det(u_{kl})} \right)_{ij} = - S,
$$
where $S$ is the scalar curvature. Notice that $(U^{ij})$ is an elliptic operator. In order to apply Shauder's estimates to control the derivatives of $u$, we need to control the derivatives of $S$ in the Euclidean sense. Using the regularity theorem, we can show that the scalar curvature $S$ has a uniform $C^k$ bound in the Euclidean sense.

\begin{prop}
Let $f$ be any smooth function on $X$ and be invariant under the torus action. Then

$$
|\partial^k f|_g^2 = \sum_{i_1,j_1,\ldots,i_k,j_k} u^{i_1 j_1} \cdots u^{i_k j_k}f_{i_1 \cdots i_k}f_{j_1 \cdots j_k}.
$$
\end{prop}

\begin{proof}
By definition,
$$
|\partial^k f|_g^2 = g^{i_1 \bar{j}_1} \cdots g^{i_k \bar{j}_k} f_{,\ i_1 \cdots i_k} f_{,\ \bar{j}_1 \cdots \bar{j}_k}.
$$
A direct calculation shows
\begin{eqnarray*}
& & g^{i_1 \bar{j}_1} \cdots g^{i_k \bar{j}_k} f_{,\ i_1 \cdots i_k} f_{,\ \bar{j}_1 \cdots \bar{j}_k} \\
&=& g^{i_k \bar{j}_k} \frac{\partial}{\partial z_{i_k}} \left(  g^{i_{k-1} \bar{j}_{k-1} } \cdots g^{i_1 \bar{j}_1} f_{,\ i_1 \cdots i_{k-1}} \right) f_{,\ \bar{j}_1 \cdots \bar{j}_k} \\
&=& g^{i_k \bar{j}_k} \frac{\partial}{\partial z_{i_k}} \left(  g^{i_{k-1} \bar{j}_{k-1} } \frac{\partial}{\partial z_{i_{k-1}}} \left( \cdots g^{i_2 \bar{j}_2} \frac{\partial}{\partial z_{i_2}} \left( g^{i_1 \bar{j}_1} f_{i_1} \right) \cdots \right) \right) f_{,\ \bar{j}_1 \cdots \bar{j}_k} \\
&=& \frac{\partial^k f}{\partial x_{j_k} \cdots \partial x_{j_1}} f_{,\ \bar{j}_1 \cdots \bar{j}_k} \\
&=& \delta_{j_1}^{i_1} \cdots \delta_{j_k}^{i_k} \frac{\partial^k f}{\partial x_{i_k} \cdots \partial x_{i_1}} f_{,\ \bar{j}_1 \cdots \bar{j}_k} \\
&=& u^{i_1 \alpha_1 } \cdots u^{i_k \alpha_k} \frac{\partial^k f}{\partial x_{i_k} \cdots \partial x_{i_1}} g^{\alpha_1 \bar{j}_1} \cdots g^{\alpha_k \bar{j}_k} f_{,\ \bar{j}_1 \cdots \bar{j}_k} \\
&=& u^{i_1 j_1} \cdots u^{i_k j_k}f_{i_1 \cdots i_k}f_{j_1 \cdots j_k}.
\end{eqnarray*}
\end{proof}

\begin{cor}
\label{reg}
If $(D^2 u)$ is bounded from above, then $D^k S$ is bounded.
\end{cor}

\begin{proof}
We choose an orthonormal basis such that $(D^2 u) = diag(\lambda_1, \ldots, \lambda_n)$. Since $(D^2 u)$ is bounded from above, we have $\lambda_i \leq C$ for all $i$.  Thus
\begin{eqnarray*}
& & u^{i_1 j_1} \cdots u^{i_k j_k}S_{i_1 \cdots i_k}S_{j_1 \cdots j_k}\\
&=& \sum_{i_1,\ldots,i_k} \frac{S_{i_1 \cdots i_k}^2}{\lambda_{i_1} \cdots \lambda_{i_k}}\\
&\geq& \frac{1}{C^k} \sum_{i_1,\ldots,i_k} S_{i_1 \cdots i_k}^2.
\end{eqnarray*}
Since $|\nabla^k S|$ is bounded by the regularity theorem, we conclude that $|D^k S|$ is bounded.
\end{proof}

We normalize $u_i(0, \cdot)$ at the origin by subtracting an affine function such that
$$
u_i(0,0) = 0, \quad D_x u_i(0,0) = 0.
$$
Since for any $t$, $\left(D^2 u_i(t, \cdot) \right)$ is bounded from above,  we obtain the $C^k$ bound of $S$ in the Euclidean sense. Thus there is a constant $\delta_0 > 0$ such that
$$
C_1 < (D^2 u_i(t, 0)) < C_2
$$
for all $-\delta_0 \leq t \leq 0$. Hence for any $p$ and $ t \in [-\delta_0, 0]$,  $(D^2 u_i(t, p))$ is bounded from below where the lower bound depends on the Euclidean distance from $p$ to the origin. So we obtain the $C^k$ bound of $u_i(t, p)$ in the Euclidean sense for $t \in [-\delta_0, 0]$. 

Next we calculate the derivatives of $u_i(t, p)$ in time and the mixed derivatives in time and space. Notice that the first derivative of $u_i(t, \cdot)$ in $t$ is just the scalar curvature $S_i(t, \cdot)$. We use $u$ instead of $u_i(t, \cdot)$ in the following calculations for convenience. The second derivative of $u$ in $t$ is
\begin{eqnarray*}
\frac{\partial^2 u}{\partial t^2} & = & - \frac{\partial S}{\partial t} \\
&=& \sum_{ij} \left( \frac{\partial u^{ij} }{\partial t} \right)_{ij} \\
&=& \sum_{ij} \left( u^{ik} S_{kl} u^{jl} \right)_{ij} \\
&=& - \left( u^{i \alpha} u_{\alpha \beta i} u^{\beta k} S_{kl} u^{jl} \right)_j + \left(u^{ik} S_{kli} u^{jl} \right)_j - \left( u^{ik} S_{kl} u^{j\alpha} u_{\alpha \beta i} u^{\beta l} \right)_j \\
&=& u^{i \gamma} u_{\gamma \delta j} u^{\delta \alpha} u_{\alpha \beta i} u^{\beta k} S_{kl} u^{jl} - u^{i \alpha} u_{\alpha \beta i j} u^{\beta k} S_{kl} u^{jl} +\\
& & u^{i \alpha} u_{\alpha \beta i} u^{\beta \gamma} u_{\gamma \delta j} u^{\delta k} S_{kl} u^{jl} - u^{i \alpha} u_{\alpha \beta i} u^{\beta k} S_{klj} u^{jl} +\\
& &  u^{i \alpha} u_{\alpha \beta i} u^{\beta k} S_{kl} u^{j\gamma} u_{\gamma \delta j} u^{\delta l} - u^{i \alpha} u_{\alpha \beta j} u^{\beta k} S_{kli} u^{jl} +\\
& & u^{ik} S_{klij} u^{jl} - u^{ik} S_{kli} u^{j\alpha} u_{\alpha \beta j} u^{\beta l} + u^{i \gamma} u_{\gamma \delta j} u^{\delta k} S_{kl} u^{j\alpha} u_{\alpha \beta i} u^{\beta l} -\\
& & u^{ik} S_{klj} u^{j\alpha} u_{\alpha \beta i} u^{\beta l} + u^{ik} S_{kl} u^{j \gamma} u_{\gamma \delta j} u^{\delta \alpha} u_{\alpha \beta i} u^{\beta l} -\\
& & u^{ik} S_{kl} u^{j\alpha} u_{\alpha \beta i j}  u^{\beta l} + u^{ik} S_{kl} u^{j\alpha} u_{\alpha \beta i} u^{\beta \gamma} u_{\gamma \delta j} u^{\delta l}. \hspace{2 cm} (*)
\end{eqnarray*}

\begin{lemma}
If we change the coordinate system by an orthonormal transformation,  the value of $(*)$ remains unchanged at the origin. 
\end{lemma}

\begin{proof}
Let $O=(a_{ij})$ be an orthonormal matrix and $v(x) = u(x ~ O), A(x) = S(x ~ O).$ Following the calculations in Claim 4.1 of \cite{H1}, we have
$$
\frac{\partial v}{\partial x_i}(0) = \sum_{\alpha} a_{i \alpha} u_\alpha(0), \quad \frac{\partial^2 v}{\partial x_i \partial x_j}(0) = \sum_{\alpha, \beta} a_{i\alpha} a_{j\beta} u_{\alpha \beta}(0), 
$$

$$
\frac{\partial^3 v}{\partial x_i \partial x_j \partial x_k}(0) = \sum_{\alpha, \beta, \gamma} a_{i\alpha} a_{j\beta} a_{k \gamma} u_{\alpha \beta \gamma}(0),
$$

$$
\frac{\partial^4 v}{\partial x_i \partial x_j \partial x_k \partial x_l}(0) = \sum_{\alpha, \beta, \gamma, \delta} a_{i\alpha} a_{j\beta} a_{k \gamma} a_{l \delta} u_{\alpha \beta \gamma \delta}(0).
$$

And

$$
\frac{\partial A}{\partial x_i}(0) = \sum_{\alpha} a_{i \alpha} S_\alpha(0), \quad \frac{\partial^2 A}{\partial x_i \partial x_j}(0) = \sum_{\alpha, \beta} a_{i\alpha} a_{j\beta} S_{\alpha \beta}(0), 
$$

$$
\frac{\partial^3 A}{\partial x_i \partial x_j \partial x_k}(0) = \sum_{\alpha, \beta, \gamma} a_{i\alpha} a_{j\beta} a_{k \gamma} S_{\alpha \beta \gamma}(0),
$$

$$
\frac{\partial^4 A}{\partial x_i \partial x_j \partial x_k \partial x_l}(0) = \sum_{\alpha, \beta, \gamma, \delta} a_{i\alpha} a_{j\beta} a_{k \gamma} a_{l \delta} S_{\alpha \beta \gamma \delta}(0).
$$

Also we have

$$
v^{ij}(0) = \sum_{\alpha, \beta} a_{i\alpha} a_{j\beta} u^{\alpha \beta}(0).
$$

By routine calculations, we can check that, for example,
\begin{eqnarray*}
u^{i \gamma} u_{\gamma \delta j} u^{\delta \alpha} u_{\alpha \beta i} u^{\beta k} S_{kl} u^{jl} = v^{i \gamma} v_{\gamma \delta j} v^{\delta \alpha} v_{\alpha \beta i} v^{\beta k} A_{kl} v^{jl}.
\end{eqnarray*}
Thus we obtain the conclusion.
\end{proof}

To show that $(*)$ is bounded, we can assume that we are in the origin. By an orthonormal transformation of the coordinate system, we can assume that $(u_{ij}) = diag (\lambda_1, \ldots, \lambda_n)$ is a diagonal matrix. Then

$$
u^{i \gamma} u_{\gamma \delta j} u^{\delta \alpha} u_{\alpha \beta i} u^{\beta k} S_{kl} u^{jl} = \sum_{i,j,k,\delta} \frac{1}{\lambda_i \lambda_j  \lambda_k \lambda_\delta} u_{i \delta j} u_{\delta k i} S_{kj}.
$$
The right hand side is obviously bounded. Thus we conclude that $\frac{\partial^2 u}{\partial t^2}(t, p), t \in [-\delta_0, 0]$ is bounded by a constant depending only on the Euclidean distance between $p$ and the origin. Similar arguments show that

\begin{prop}
\label{time derivative}
$$
\frac{\partial^k u}{\partial t^k}(t, p), \quad \frac{\partial^{k+l} u}{\partial t^k \partial x^l}(t, p)
$$ is bounded for all $t \in [-\delta_0, 0], k \geq 1, l \geq 0$. The bounds depend only on $k$ or $k, l$ and the Euclidean distance between $p$ and the origin.
\end{prop}

So there is a subsequence of $u_i(t)$ converging to a limiting Calabi flow $u_\infty(t), -\delta_0 \leq t \leq 0$. Moreover, $u_\infty(0)$ is a smooth convex function $\bar{u}$ in $\mathbb{R}^n$ with the following property:
\begin{enumerate}
\item The $L^\infty$ norm of Riemannian curvature is bounded by 1, i.e., $$|Rm|_{L^\infty}=\max_{p \in \mathbb{R}^n}\sum_{i, j} \bar{u}^{ij}_{\ kl}(p) \bar{u}^{kl}_{\ ij}(p) \leq 1.$$
\item The Euclidean derivative of $\bar{u}$ is bounded by $M$ which is the same constant in the $M$-condition, i.e.,
$$
|D \bar{u}| < M.
$$
\end{enumerate}

We want to rule out the above singularity by the following nonexistence lemma:

\begin{lemma}
If in addition $\bar{S} = 0$, then such $\bar{u}$ cannot exist.
\end{lemma}

\begin{proof}
For dimensional 2, see Theorem 2 in \cite{D3}. For higher dimension, see Proposition 5.2 in \cite{H1}.
\end{proof}

\begin{rmk}
In dimensional 2, Jia and Li prove a more general result in \cite{JL}: the solution to the equation
$$
\sum_{ij} u^{ij}_{\ ij} = 0
$$
in $\mathbb{R}^2$ must be a quadratic function.
\end{rmk}

\subsection{The case of  dimension 2}

In dimension 2, we prove that the scalar curvature of the limiting Calabi flow $u_\infty(t)$ is 0 . For each $u_i(t)$, we temporarily suppress the index $i$. Let $$E^{ij} = \frac{\partial u^{ij}}{\partial t}$$ and notice that the derivative of the Calabi energy
$$
Ca_i (t) = \int_{P_i} S(t,x)^2 dx
$$
with respect to $t$ is
\begin{eqnarray*}
& &\frac{\partial}{\partial t} \int_{P_i} S(t,x)^2 dx\\
& = & - 2\int_{P_i} (S - \underline{S}) E^{ij}_{\ ij} dx \\
& = & -2\int_{\partial P_i} (S - \underline{S}) E^{ij}_{\ i} \nu_j ds+ 2 \int_{P_i} S_j E^{ij}_{\ i} dx \\
& = & 2 \int_{P_i} S_j E^{ij}_{\ i} dx\\
& = & 2 \int_{\partial P_i} S_j E^{ij} \nu_i ds - 2 \int_{P_i} S_{ij}E^{ij} dx\\
& = & -2 \int_{P_i} S_{ij} E^{ij} d \mu \\
& = & -2 \int_{P_i} S_{ij} u^{ia} S_{ab} u^{bj} dx \\
&\leq& 0.\\
\end{eqnarray*}
Hence
$$
Ca_i(-\delta_0) - Ca_i(0) = 2 \int_{-\delta_0}^0 \int_{P_i} S_{ij} u^{ia} S_{ab} u^{bj} \ dx \ dt.
$$

For the limiting Calabi flow, we have
$$
0 = \lim_{i \rightarrow \infty}Ca_i(-\delta_0) - Ca_i(0) \geq 2 \int_{-\delta_0}^0 \int_{\mathbb{R}^n}S_{ij} u^{ia} S_{ab} u^{bj} \ dx \ dt.
$$
That is to say, $S(t)$ must be an affine function on $\mathbb{R}^n$. Since
$$
Ca_\infty(0) = \lim_{i \rightarrow \infty} Ca_i(0) < C,
$$
$S(t)$ must be 0. 

\begin{proof}[Proof of Theorem \ref{long}]
In \cite{Ca1}, Calabi shows that the Calabi flow decreases the Calabi energy. Moreover, he shows that in dimension 2, the total energy is equivalent to the Calabi energy. Combining the above results, we obtain that the curvature is uniformly bounded along the Calabi flow.
\end{proof}

\section{Exponential convergence}

Suppose that the Calabi flow exists for all time and the curvature is uniformly bounded. The remaining question is whether the Calabi flow converges to the flat K\"ahler metric. A well-known fact is that the Calabi flow decreases the Mabuchi energy. In our case, the Mabuchi energy can be explicitly written as
$$
Ma(u) = - \int_P \log \det(u_{ij}) \ dx.
$$

Taking the derivative with respect to the time variable $t$, we obtain
$$
\frac{\partial}{\partial t} Ma(u(t,x)) = \int_P u^{ij}(t,x) S_{ij}(t,x) \ dx = - \int_P S^2(t,x) \ dx  \leq 0.
$$

Since the curvature is bounded uniformly and the $M$-condition is preserved along the Calabi flow, we know that $(u_{ij}(t,x))$ is bounded from above. The fact that the Mabuchi energy is decreasing along the Calabi flow shows that for any $t$, there is at least one point $p \in P$ such that
$$
\det(D^2 u(t,x) (p)) > C
$$
for some constant $C$. Hence
$$
(D^2 u(t,x) (p)) > C I_n.
$$
Thus $(D^2 u(t,x))$ is bounded from below point-wisely. Notice that $u(t,x)$ has a priori $C^0$ and $C^1$ bound. Hence if we take a sequence of time $t_i \rightarrow \infty$, applying Corollary (\ref{reg}) and Proposition (\ref{time derivative}), we can show that there is a subsequence of $t_i$ such that the Calabi flow
$$u_i(t,x) = u(t- t_i, x)$$ in the interval $[-1,0]$ converges to a limiting flow $u_\infty (t,x), t \in [-1,0]$. Following the arguments as in the previous section, we conclude that $A_\infty (t,x)$ must be an affine function with respect to $x$. Hence $u_\infty (t,x)$ must be the flat K\"ahler metric.

To show the exponential convergence, by the stability result of \cite{HZ}, we only need to show that the corresponding K\"ahler potential $\phi(t_i, \xi)$ satisfies the following conditions:

\begin{itemize}
\item $$\omega(t_i) \geq C_1 \omega_\infty.$$

\item $$|\phi(t_i, \xi)|_{C^{2,\alpha}(\omega_\infty)} < C_2.$$

\item $$\lim_{i \rightarrow \infty} dist(\phi(t_i), \phi_\infty) = 0.$$
\end{itemize}

It is easy to see that the first and the last conditions are satisfied. Since $\omega_\infty$ is a flat metric, to control the $C^{2, \alpha}$ norm of $\phi(t_i,\xi)$. By compact embedding $C^3 \hookrightarrow C^{2,\alpha}$, we only need to control the third derivative of $\phi(t_i,\xi)$,
$$
\frac{\partial^3 \phi(t_i,\xi)}{\partial \xi_j \partial \xi_k \partial \xi_l} < C.
$$
We can obtain this inequality by the following formula
\begin{eqnarray*}
& &\frac{\partial^3 \phi(t_i,\xi)}{\partial \xi_j \partial \xi_k \partial \xi_l}\\
 &=& \frac{\partial u^{jk}(t_i,x) }{\partial \xi_l}\\
& = & \frac{\partial x_\alpha}{\partial \xi_l}  \frac{\partial u^{jk}(t_i,x) }{\partial x_\alpha} \\
&=& u^{\alpha l}(t_i,x)  \frac{\partial u^{jk}(t_i,x) }{\partial x_\alpha} \\
&=& - u^{j \alpha}(t_i,x) u^{k \beta}(t_i,x) u^{l \gamma}(t_i,x) u(t_i,x)_{\alpha \beta \gamma}.
\end{eqnarray*}

\appendix
\section{}
In this appendix, we want to show that the $M$-condition with the weak regularity theorem can give us the regularity theorem.
\begin{thm}
Suppose in $P$, the $L^\infty$ norm of Riemann curvature tensor is bounded by 1 and the symplectic potential $u$ satisfies the $M$-condition. If
$$
\int_P |\nabla^k Rm|^2 (x) dx < C(k),
$$
for all $k$, then
$$
|\nabla^k Rm|(x) < C(k),
$$
for all $k$ and $x \in P$.
\end{thm}

\begin{proof}
Let $F^k (x) = |\nabla^k Rm|(x)$. Since the curvature is bounded and $u$ satisfies the $M$-condition, we conclude that
$$
(u_{ij} (x)) < C,
$$
for all $x \in P$.
It is easy to see that for any $x \in P$,
$$
|\nabla F^k(x)| \leq |F^{k+1}(x)|.
$$
Since
$$
|\nabla F^k(x)|^2 =2 \sum_{i, j} u^{ij} F^k_i F^k_j \geq C |\nabla_E F^k|_E^2,
$$
where $ |\nabla_E F^k|_E$ is the Euclidean norm of the Euclidean derivative. Thus we have
$$
\int_P |\nabla_E F^k|_E^2 \ dx < C(k).
$$
The Sobolev embedding theorem tells us that
$$
\int_P (F^k(x))^q \ dx < C(k),
$$
where $q = \frac{2n}{n-2}$ if $n > 2$ or $q = 4$ if $n=2$. It is easy to see that after finite steps, we could reach the conclusion.
\end{proof}

Renjie Feng, \ renjie@math.northwestern.edu

Mathematics Department

Northwestern University\\

Hongnian Huang, \ hnhuang@gmail.com

CMLS

Ecole Polytechnique
\end{document}